\documentclass[12pt]{amsart}
\usepackage{amssymb}
\usepackage[usenames,dvipsnames]{xcolor}
\usepackage[colorlinks,citecolor=blue,linkcolor=BrickRed]{hyperref}
\usepackage{enumerate}
\usepackage{amsfonts,latexsym,graphicx,epsfig,amssymb,color}
\usepackage{amsmath,amstext,url,setspace}
\usepackage{amstext}    % defines the \text command, needed here
\usepackage{array}      % for advanced column specification: use >{\command} for
\usepackage{fancyhdr}
\usepackage{bbm}
\usepackage{tikz}
\usepackage{tkz-graph}
\usepackage{tkz-berge}
\usetikzlibrary{decorations.pathmorphing}
\usetikzlibrary{positioning}
\usetikzlibrary{patterns}
\usetikzlibrary{calc}
\newcommand{\C}{\mathcal{C}}
\newcommand{\vC}{\vec{\mathcal{C}}}
\newcommand{\vG}{\vec G}
\newcommand{\cA}{\mathcal{A}}
\newcommand{\bZ}{\mathbbm{Z}}

\newtheorem{theorem}{Theorem}[section]

\newtheorem{lemma}[theorem]{Lemma}

\newtheorem{claim}{Claim}[theorem]
\newtheorem{newclaim}[theorem]{Claim}
\theoremstyle{definition}

\DeclareMathOperator{\NS}{NS}

\newcommand{\LProofLabel}{}
\newcommand{\LProofType}{}

\newcommand\restr[2]{{% we make the whole thing an ordinary symbol
  \left.\kern-\nulldelimiterspace % automatically resize the bar with \right
  #1 % the function
  \vphantom{\big|} % pretend it's a little taller at normal size
  \right|_{#2} % this is the delimiter
  }}

\newcommand{\MakeCircle}[2]{% Diameter; nodes
  \pgfmathsetmacro{\angle}{360/#2};%
  \draw [thick] (0,0) circle (#1);

  \foreach \i in {1,...,#2} {
    \begin{scope}[rotate=-\i * \angle]
    \node [inner sep=0pt, minimum size=0pt, outer sep=0pt]  (p\i) at (0,#1) {};
    \end{scope}
  }
}
\makeatletter
\renewcommand\subsection{\@startsection{subsection}{1}%
  \z@{.7\linespacing\@plus\linespacing}{.5\linespacing}%
  {\normalfont\scshape\centering}}
\makeatother
\usepackage{algorithm}
\usepackage{algpseudocode}
\usepackage{biblatex}
\bibliography{excluded}

\begin{document}
\author{Jim Geelen}
\email{jfgeelen@uwaterloo.ca}
\address{Department of Combinatorics and Optimization, University of Waterloo, Waterloo,
 ON, Canada}
\author{Edward Lee}
\email{e45lee@uwaterloo.ca}
\address{Department of Combinatorics and Optimization, University of Waterloo, Waterloo,
 ON, Canada}
\title{Naji's characterization of circle graphs}
\subjclass[2010]{05C62}
\begin{abstract}
We present a simpler proof of Naji's characterization of circle graphs.
\end{abstract}
\date{\today}
\thanks{This research was partially supported by grants from the
Office of Naval Research [N00014-10-1-0851] and NSERC [203110-2011]
as well as an NSERC Canada Graduate Scholarship.}

\maketitle
\section{Introduction}
A {\it circle graph} is the intersection graph of a finite set of chords on a circle.
This class of graphs has surprising connections with planar graphs;
De Fraysseix~[\ref{FRAY}] showed that a bipartite graph
is a circle graph if and only if
it is the fundamental graph of a planar graph.  By De Fraysseix's result any characterization
of the class of circle graphs gives a characterization for the class of planar graphs.

Naji~[\ref{NG-ORIG}] gave the following beautiful characterization of the class of circle graphs
by a system of linear equations over the two-element field $\mathbb{F}_2$; one attractive feature of this characterization is that
it immediately gives an efficient algorithm for recognizing circle graphs.
\begin{theorem}[Naji's Theorem]
  A graph is a circle graph if and only if there exist values $\beta(u,v)\in\mathbb{F}_2$
  for each distinct pair $(u,v)$ of vertices  such that
  \begin{enumerate}
    \item $\beta(v, w) + \beta(w, v) = 1$ for each edge $vw$,
    \item $\beta(x, v) + \beta(x, w) = 0$ for each triple $(x,v,w)$ of distinct vertices  such that  $vw$ is an edge but $xv$ and $xw$ are not, and
    \item $\beta(v, w) + \beta(w,v) + \beta(x,v) + \beta(x, w) = 1$ for each triple $(x,v,w)$ of distinct vertices  such that  $xv$ and $xw$ are edges but $vw$ is not.
     \end{enumerate}
\end{theorem}

The specialization of Naji's Theorem to bipartite graphs gives a characterization of planar graphs. In fact, more generally, it characterizes when a binary matroid
is the cycle matroid of a planar graph. Geelen and Gerards~[\ref{GEL}] extended that specialization by characterizing when a binary matroid is graphic.

Naji's original proof, which appears in his doctoral dissertation, is very long and was not published in a journal.
Gasse~[\ref{GASSE}] published a short proof of Naji's Theorem, but it relies on Bouchet's excluded-vertex minor characterization 
of the class of circle graphs~[\ref{BO}] which is itself long and difficult.
On the other hand, Geelen and Gerards gave a very short and intuitive proof of their characterization of graphic matroids.
Motivated by that proof, Traldi~[\ref{TLN}] gave a shorter proof of Naji's Theorem, but unlike the proof in~[\ref{GEL}], Traldi's proof is 
not self-contained, relying on Bouchet's analogue of Tutte's Wheels Theorem for vertex-minors; see~[\ref{Bouchet1987}]. We give a self-contained 
proof based on the methods presented in~[\ref{GEL}].

\section{Overview}

We refer to the system of equations in Naji's Theorem as the
{\em Naji system} for the graph. For adjacent vertices $v$ and $w$, we 
denote the equation  $\beta(v, w) + \beta(w, v) = 1$ by $\NS_1(v,w)$.
For adjacent vertices $v$ and $w$ and a vertex $x$ adjacent to neither $v$ nor $w$, we denote the
equation $\beta(x, v) + \beta(x, w) = 0$ by $\NS_2(x,v,w)$.
For distinct non-adjacent vertices $v$ and $w$ and a vertex $x$ adjacent to both $v$ and $w$, we denote the
equation $\beta(v, w) + \beta(w,v) + \beta(x,v) + \beta(x, w) = 1$  by $\NS_3(x,v,w)$.

A {\it chord diagram} is a drawing of a circle and some chords with disjoint ends.
A {\it circle graph} is the intersection graph of the chords of some chord diagram;
see Figure \ref{label:circlegraph}.  

\begin{figure}[h!]
  \begin{tikzpicture}[scale=0.3]
    \MakeCircle{4}{10}
    \draw [thick] (p1) -- (p4);
    \draw node [inner sep=0pt, outer sep=0pt, minimum size=1pt,above right=2pt of p1] {1};
    \draw node [inner sep=0pt, outer sep=0pt, minimum size=1pt,below right=2pt of p4] {1};
    \draw [thick] (p9) -- (p2);
    \draw node [inner sep=0pt, outer sep=0pt, minimum size=1pt,above left=2pt of p9] {5};
    \draw node [inner sep=0pt, outer sep=0pt, minimum size=1pt,right=2pt of p2] {5};
    \draw [thick] (p3) -- (p6);
    \draw node [inner sep=0pt, outer sep=0pt, minimum size=1pt,right=2pt of p3] {2};
    \draw node [inner sep=0pt, outer sep=0pt, minimum size=1pt,below left=2pt of p6] {2};
    \draw [thick] (p5) -- (p8);
    \draw node [inner sep=0pt, outer sep=0pt, minimum size=1pt,below=2pt of p5] {3};
    \draw node [inner sep=0pt, outer sep=0pt, minimum size=1pt,left=2pt of p8] {3};
    \draw [thick] (p7) -- (p10);
    \draw node [inner sep=0pt, outer sep=0pt, minimum size=1pt,below left=2pt of p7] {4};
    \draw node [inner sep=0pt, outer sep=0pt, minimum size=1pt,above=2pt of p10] {4};
  \end{tikzpicture}
  \hspace{10mm}
  \begin{tikzpicture}[scale=1.2]
    \Vertices{circle}{1,5,4,3,2}
    \Edges(1,2,3,4,5,1)
  \end{tikzpicture}
  \caption{A Chord Diagram and its Circle Graph}
  \label{label:circlegraph}
\end{figure}
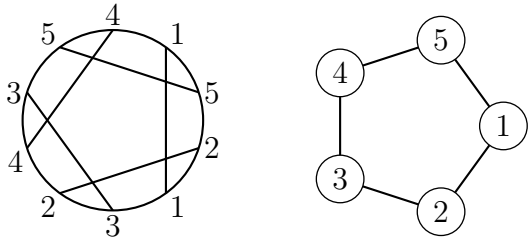

We start by showing how to construct a solution to the Naji system when $G$ is a circle graph.
Let $\C$ be a chord diagram for a circle graph $G=(V,E)$ and let $\vec\C$ be obtained from $\C$ by giving 
each chord an orientation. For a pair $(v,w)$ of distinct vertices of $G$ we define $\beta_{\vec\C}(v,w)=0$ if
the head of the chord $w$ is ``to the right" of the chord $v$ in $\vec\C$ (that is, we encounter the head of $w$ when we travel clockwise from the 
head of $v$ to the tail of $v$); see Figure~\ref{figure:solution:set}.
Otherwise, when the head of $w$ is to the left of $v$, we define $\beta_{\vec\C}(v,w)=1.$
Note that, if $u$ and $v$ are intersecting chords in $\vC$, then exactly one of $u$ and $v$
crosses the other from left to right; so $\beta_{\vC}$ satisfies $\NS_1(u,v)$. Moreover, if $x$ 
is a third chord that crosses neither $u$ nor $v$, then the heads of $u$ and $v$ lie on the same side of $x$; 
so $\beta_{\vC}$ satisfies $\NS_2(x,u,v)$.  
Finally consider three distinct chords $x$, $u$, and $v$ where $x$ intersects both $u$ and $v$ but $u$ and $v$ do
not intersect. Note that the heads of $u$ and $v$ are on the same side of $x$ if and only if exactly one of $u$ 
and $v$ is to the left of the other; thus
\[\beta_{\vC}(x,u)+\beta_{\vC}(x,v) = \beta_{\vC}(u,v)+\beta_{\vC}(v,u)+1\] and, hence, $\beta_{\vC}$ satisfies $\NS_3(x,u,v)$.

\begin{figure}
	\begin{tikzpicture}[scale=0.3]
		\MakeCircle{4}{12}
		\draw [->, line width = 2] (p6) -- (p12);
		\draw [->, line width = 2] (p9) -- (p3);
    	\draw node [inner sep=0pt, outer sep=0pt, minimum size=1pt,above=2pt of p12] {$v$};
    	\draw node [inner sep=0pt, outer sep=0pt, minimum size=1pt,right=2pt of p3] {$w$};	
	\end{tikzpicture}
	\begin{tikzpicture}[scale=0.3]
		\MakeCircle{4}{12}
		\draw [->, line width = 2] (p6) -- (p12);
		\draw [->, line width = 2] (p1) -- (p4);
    	\draw node [inner sep=0pt, outer sep=0pt, minimum size=1pt,above=2pt of p12] {$v$};
    	\draw node [inner sep=0pt, outer sep=0pt, minimum size=1pt,below right=2pt of p4] {$w$};
	\end{tikzpicture}
	\begin{tikzpicture}[scale=0.3]
		\MakeCircle{4}{12}
		\draw [->, line width = 2] (p6) -- (p12);
		\draw [->, line width = 2] (p4) -- (p1);
    	\draw node [inner sep=0pt, outer sep=0pt, minimum size=1pt,above=2pt of p12] {$v$};
    	\draw node [inner sep=0pt, outer sep=0pt, minimum size=1pt,below right=2pt of p4] {$w$};
	\end{tikzpicture}
	\caption{When $\beta(v, w) = 0$.}
	\label{figure:solution:set}
\end{figure}
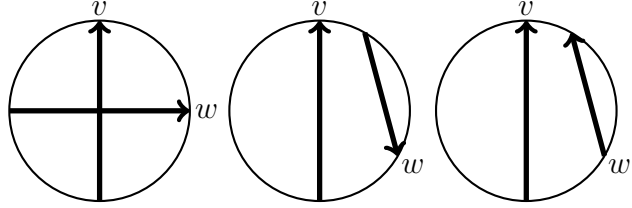

So, if $G$ is a circle graph, then the Naji system has a solution. To verify the converse we consider a solution $\beta$ to the Naji system of a graph $G$.
If $\vC$ is an oriented chord diagram for $G$ and $\beta=\beta_{\vec\C}$ then we say that $\vC$ is an {\em oriented chord diagram for $(G,\beta)$}.
We say that $\beta$ is {\em chordal} if there is an oriented chord diagram for $(G,\beta)$.
Unfortunately not all solutions to Naji systems are chordal; for example, 
both the complete graph $K_4$ and the Claw $K_{1,3}$, depicted in Figure~\ref{fig:splits}, admit non-chordal solutions, 
as shown in the following two tables:
\[
\left.\begin{array}{c|cccc} \beta & a & b & c & d \\ \hline a &  & 1 & 0 & 0 \\ b  & 0 & & 1 & 0 \\ c & 1 & 0 & & 0 \\ d & 1 & 1 & 1  \end{array}\right.
\hspace{2cm}
\left.\begin{array}{c|cccc} \beta & x & a & b & c \\ \hline x &  & 1 & 1 & 1 \\ a  & 0 & & 0 & 1 \\ b & 0 & 1 & & 0 \\ c & 0 & 0 & 1  \end{array}\right.
\]

Note that, if $\beta$ is chordal, then the restriction of $\beta$ to any induced subgraph of $G$ is also chordal.
We will call an induced subgraph $H$ of $G$ an {\em obstruction} if the restriction of $\beta$ to $H$
is not chordal. We prove the following result in the next section; it shows that the Claw and $K_4$ are the minimal connected obstructions.
\begin{lemma}\label{obstructions}
Let $\beta$ be a solution to the Naji system for a graph $G$.
If $\beta$ is not chordal and $G$ is connected, then there is an obstruction that is isomorphic to the Claw or to $K_4$.
\end{lemma}

We will use the terms {\em $K_4$-obstruction} and {\em Claw-obstruction} to refer to obstructions isomorphic
to $K_4$ and the Claw respectively. 

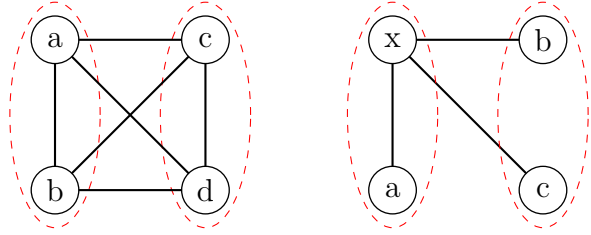
\begin{figure}
\begin{tikzpicture}
    \Vertex[x=-1,y=1]{a}
 \Vertex[x=-1,y=-1]{b}
 \Vertex[x=1,y=-1]{d}
 \Vertex[x=1,y=1]{c}
     \draw[red, dashed] (-1,0) ellipse (6mm and 15mm);
    \draw[red, dashed] (1,0) ellipse (6mm and 15mm);
   \Edge(a)(d)
   \Edge(b)(d)
   \Edge(c)(d)
   \Edge(b)(a)
   \Edge(c)(a)
   \Edge(c)(b)
\end{tikzpicture}
\hspace{10mm}
\begin{tikzpicture}
 %   \Vertices{circle}{a, b, c, d}
 \Vertex[x=-1,y=1]{x}
 \Vertex[x=-1,y=-1]{a}
 \Vertex[x=1,y=-1]{c}
 \Vertex[x=1,y=1]{b}
    \Edge(x)(a)
    \Edge(x)(b)
    \Edge(x)(c)
    \draw[red, dashed] (-1,0) ellipse (6mm and 15mm);
    \draw[red, dashed] (1,0) ellipse (6mm and 15mm);
\end{tikzpicture}
\caption{Splits in $K_4$ and the Claw}
\label{fig:splits}
\end{figure}

A {\em split} in a graph $G$ is a partition $(X,Y)$ of $V(G)$ such that $X$ and $Y$ each have at least two vertices 
and the set of edges connecting $X$ to $Y$ induces a complete bipartite graph. Note that there are three ways to partition
four vertices into pairs and each of these partitions is a split both in $K_4$ and in the Claw; see Figure~\ref{fig:splits}.
The hardest part of our proof is showing that splits in $K_4$- or Claw-obstructions extend to splits in $G$; see Section~\ref{section:splits} for details.
\begin{lemma}\label{induced:split}
Let $\beta$ be a non-chordal solution to the Naji system for a graph $G$, let $H$ be a
$K_4$-obstruction or a Claw-obstruction in $G$, and let $(X,Y)$ be a split in $H$. Then 
there is a split $(X',Y')$ in $G$ such that $X\subseteq X'$ and $Y\subseteq Y'$.
\end{lemma}

In our proof of Lemma~\ref{induced:split} we consider $K_4$-obstructions
and Claw-obstructions independently, but readers familiar with the paper of Gasse~[\ref{GASSE}] 
will observe that the two cases are in fact equivalent under local complementation.

Now we prove Naji's Theorem as a consequence of Lemmas~\ref{obstructions} and~\ref{induced:split}.
\begin{proof}[Proof of Naji's Theorem]
Consider a counter-example $G=(V,E)$ with $|V|$ minimum. Thus $G$ is not a circle graph
but there is a solution $\beta$ to the Naji system for $G$. Since $G$ is not a circle graph,
$\beta$ is not chordal. By minimality, $G$ is connected. 
Then, by  Lemmas~\ref{obstructions} and~\ref{induced:split},
the graph $G$ has a split $(X,Y)$. 
Let $xy$ be an edge with $x\in X$ and $y\in Y$. Let $G_1$ and $G_2$ denote the 
subgraphs of $G$ induced by $X\cup \{y\}$ and $Y\cup\{x\}$ respectively; see Figure~\ref{figure:shrink}.

Since $G_1$ and $G_2$ both have fewer vertices than $G$, by minimality, both $G_1$ and $G_2$ 
are circle graphs. By a well known folklore result printed by Bouchet in [\ref{Bouchet1987}], we can construct a chord diagram
for $G$ by composing chord diagrams for  $G_1$ and $G_2$ together as shown in Figure~\ref{figure:shrink-chord},
contradicting the fact that $G$  is not itself a circle graph.
\end{proof}

\begin{figure}
\begin{center}
  \begin{tikzpicture}[scale=0.75]
    \GraphInit[vstyle = Simple]
    \SetVertexSimple[MinSize = 3pt]
    \begin{scope}[shift={(0,1.8)}]
      \Vertices[x=-0.8,y=-1,dir=\NO]{line}{1,2,3}
      \Vertices[x=0.8,y=-1,dir=\NO]{line}{4,5,6}
      \draw (-1.5,0) ellipse (1.2 and 1.5);
      \draw (1.5,0) ellipse (1.2 and 1.5);
      \Edge(1)(4) \Edge(1)(5) \Edge(1)(6)
      \Edge(2)(4) \Edge(2)(5) \Edge(2)(6)
      \Edge(3)(4) \Edge(3)(5) \Edge(3)(6)
      \node at (-1.5,-1.1) {$X$};
      \node at (1.5,-1.1) {$Y$};
    \end{scope}
    \begin{scope}[shift={(-3,-1.8)}]
      \Vertices[x=-0.8,y=-1,dir=\NO]{line}{1a,2a,3a}
      \Vertex[x=0.8,y=0]{4a}
      \draw (-1.5,0) ellipse (1.2 and 1.5);
      \Edge(1a)(4a) \Edge(2a)(4a) \Edge(3a)(4a)
      \node at (-1.5,-1.1) {$X$};
      \node at (1,-.5) {$y$};
    \end{scope}
    \begin{scope}[shift={(3,-1.8)}]
      \Vertex[x=-0.8,y=0]{2b}
      \Vertices[x=0.8,y=-1,dir=\NO]{line}{4b,5b,6b}
      \draw (1.5,0) ellipse (1.2 and 1.5);
      \Edge(2b)(4b) \Edge(2b)(5b) \Edge(2b)(6b)
      \node at (-1,-.5) {$x$};
      \node at (1.5,-1.1) {$Y$};
    \end{scope}
  \end{tikzpicture}
  \caption{$G$, $G_1$, and $G_2$}
  \label{figure:shrink}
\end{center}
\end{figure}
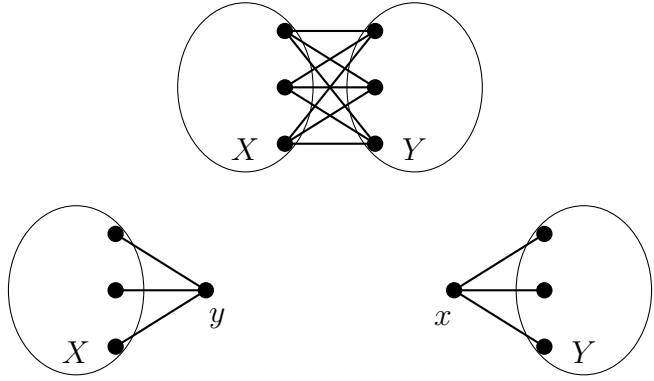

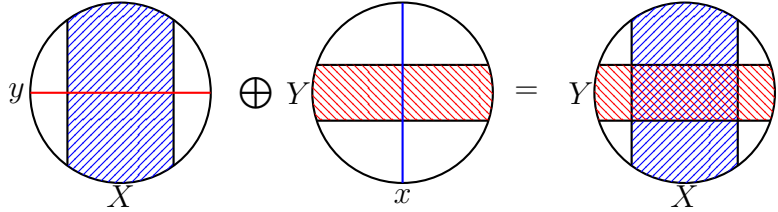
\begin{figure}
\begin{center}
  \begin{tikzpicture}[scale=0.3]
    \begin{scope}[shift={(9,0)}]
      \MakeCircle{4}{10}
      \draw [thick] (p9) -- (p6);
      %\draw node [inner sep=0pt, outer sep=0pt, minimum size=1pt,above left=2pt of p9] {$a_1$};   %p9
      %\draw node [inner sep=0pt, outer sep=0pt, minimum size=1pt,below right=2pt of p4] {$a'_1$};          %p2
      \draw [thick] (p4) -- (p1);
      %\draw node [inner sep=0pt, outer sep=0pt, minimum size=1pt,above right=2pt of p1] {$a_k$};          %p3
      %\draw node [inner sep=0pt, outer sep=0pt, minimum size=1pt,below left=2pt of p6] {$a'_m$};   %p6
      \draw [thick] (p2) -- (p8);
      %\draw node [inner sep=0pt, outer sep=0pt, minimum size=1pt,above right=2pt of p2] {$b_1$};        %p5
      %\draw node [inner sep=0pt, outer sep=0pt, minimum size=1pt,below left=2pt of p7] {$b'_1$};            %p8
      \draw [thick] (p3) -- (p7);
      %\draw node [inner sep=0pt, outer sep=0pt, minimum size=1pt,below right=2pt of p3] {$b_l$};   %p7
      %\draw node [inner sep=0pt, outer sep=0pt, minimum size=1pt,above left=2pt of p8] {$b'_n$};      %p10
      \node at (0,-4.6) {$X$};
      \node at (-4.6,0) {$Y$};
      \clip circle (4);
      \draw[pattern = north east lines, pattern color = blue] ([shift={(0,1)}]p9) rectangle ([shift={(0,-1)}]p4);
      \draw[pattern = north west lines, pattern color = red] ([shift={(-1,0)}]p8) rectangle ([shift={(1,0)}]p3);
    \end{scope}
    \begin{scope}[shift={(-16,0)}]
      \MakeCircle{4}{10}
      \draw [thick] (p9) -- (p6);
      %\draw node [inner sep=0pt, outer sep=0pt, minimum size=1pt,above left=2pt of p9] {$a_1$};   %p9
      %\draw node [inner sep=0pt, outer sep=0pt, minimum size=1pt,below right=2pt of p4] {$a'_1$};          %p2
      \draw [thick] (p4) -- (p1);
      %\draw node [inner sep=0pt, outer sep=0pt, minimum size=1pt,above right=2pt of p1] {$a_k$};          %p3
      %\draw node [inner sep=0pt, outer sep=0pt, minimum size=1pt,below left=2pt of p6] {$a'_m$};   %p6
      \node at (0,-4.6) {$X$};
      \node at (-4.6,0) {$y$};
      \clip circle (4);
      \draw[pattern = north east lines, pattern color = blue] ([shift={(0,1)}]p9) rectangle ([shift={(0,-1)}]p4);
      \draw[thick, color=red] (-4,0) -- (4,0);
    \end{scope}
    \begin{scope}[shift={(-3.5,0)}]
      \MakeCircle{4}{10}
      \draw [thick] (p2) -- (p8);
      %\draw node [inner sep=0pt, outer sep=0pt, minimum size=1pt,above right=2pt of p2] {$b_1$};        %p5
      %\draw node [inner sep=0pt, outer sep=0pt, minimum size=1pt,below left=2pt of p7] {$b'_1$};            %p8
      \draw [thick] (p3) -- (p7);
      %\draw node [inner sep=0pt, outer sep=0pt, minimum size=1pt,below right=2pt of p3] {$b_l$};   %p7
      %\draw node [inner sep=0pt, outer sep=0pt, minimum size=1pt,above left=2pt of p8] {$b'_n$};      %p10
      \node at (0,-4.6) {$x$};
      \node at (-4.6,0) {$Y$};
      \clip circle (4);
      \draw[pattern = north west lines, pattern color = red] ([shift={(-1,0)}]p8) rectangle ([shift={(1,0)}]p3);
      \draw[thick, color=blue] (0,4) -- (0,-4);
    \end{scope}
    \node at (2,0) {$\displaystyle{=}$};
    \node at (-10,0) {$\bigoplus$};
  \end{tikzpicture}
  \caption{Chord diagrams for $G_1$, $G_2$ and $G$}
  \label{figure:shrink-chord}
\end{center}
\end{figure}

\section{Finding the obstructions}

The goal of this section is to prove Lemma~\ref{obstructions}.
We need the following strengthening of the Naji equations of type $(2)$.
\begin{lemma}\label{strong}
Let $\beta$ be a solution to the Naji system for a graph $G$ and let $x$, $u$, and $v$
be distinct vertices. If there is a path $P$ from $u$ to $v$ and 
$x$ is not adjacent to any vertex in $P$, then $\beta(x,u)=\beta(x,v)$.
\end{lemma}

\begin{proof} By possibly taking shortcuts, we may assume that the path $P$ is an induced subgraph of $G$.   
Suppose that the vertices of $P$ are $(v_0,v_1,\ldots,v_k)$, in that order.   
For each $i\in \{0,\ldots,k-1\}$, the equation $\NS_2(x,v_i,v_{i+1})$ gives
$\beta(x,v_{i})=\beta(x,v_{i+1})$. Hence $\beta(x,u)=\beta(x,v)$, as required.
\end{proof}

We start by describing an equivalence operation on solutions to the Naji system; this operation essentially corresponds 
to reorienting chords, but it also applies to non-chordal solutions.

Let $\vec\C_1$ be an oriented chord diagram for a circle graph $G$ and let $\vec\C_2$ denote a second chord
diagram obtained from $\vec\C_1$ by changing the orientation on a single chord $c$. For distinct vertices
$u$ and $v$ we have
$\beta_{\vec\C_1}(u,v) \neq\beta_{\vec\C_2}(u,v)$ if and only if
either $c=u$ or $c=v$ and $uv$ is an edge.

Now consider an arbitrary graph $G$. For three vertices $c,u,v\in V$ with $u\neq v$ we define $\delta_c(u,v)=1$
if either $c=u$ or $c=v$ and $uv$ is an edge; otherwise we define $\delta_c(u,v)=0$.
The following result, due to Gasse~[\ref{GASSE}], shows that the equivalence that we saw above on 
chordal solutions extends to arbitrary solutions of the Naji system; the proof can be verified by an easy case check.
\begin{lemma} \label{lemma:swap}
If $\beta$ be a solution to the Naji system of a graph $G$ and $c$ is a vertex of $G$, then
$\beta+\delta_c$ is also a solution to the Naji system.
\end{lemma}
We say that $\beta+\delta_c$ is obtained from $\beta$ by {\em reorienting} $c$.

Let $\beta$ be a solution to the Naji system for a graph $G=(V,E)$. Our goal is to find either a $K_4$-obstruction or a 
Claw-obstruction in the case that $\beta$ is not chordal. 

For a set $X$ of vertices, we let $\beta[X]$ denote the restriction of $\beta$ to $X$. Suppose that 
$\beta[V-c]$ is chordal for some $c\in V$, and let $\vec\C$ be an oriented chord diagram 
for $(G-c,\beta[V-c])$. Consider trying to extend $\vec\C$ to an oriented 
chord diagram for $(G,\beta)$. For a vertex $v\in V-c$, if $\beta(v,c)=0$, then 
we want to place the head of $c$ to the right of $v$ and if $\beta(v,c)=1$ then we wish to place
the head of $c$ to the left of $v$. 

Let $H(\beta,\vec\C,v)$ denote the  open arc of the circle into which the head of
$c$ is required to be placed relative to $v$. Therefore the head of $c$ is required to go into the intesection of
the sets $H(\beta,\vec\C,v)$ where $v$ ranges over all vertices in $V-c$; we denote this intersection
by $H(\beta,\vec\C)$. 

To determine the position of the tail of $c$ we simply reorient $\beta$ at $c$; thus we define
\[T(\beta,\vec\C,v)=H(\beta+\delta_c,\vec\C,v), \text{ and}\] 
\[T(\beta,\vec\C)=H(\beta+\delta_c,\vec\C).\]

\begin{lemma} \label{lemma:extend}
Let $\beta$ be a solution for the Naji system of a graph $G=(V,E)$, let $c\in V$, and let $\vec\C$ be an oriented 
chord diagram for $(G-v, \beta[V-c])$. Now let $t\in T(\beta,\vec\C)$, $h\in H(\beta,\vec\C)$,
and let $\vec\C_1$ denote the oriented chord diagram obtained from $\vec\C$ by adding $c$ as an oriented chord with tail
$t$ and head $h$. Then $\vec\C_1$ is an oriented chord diagram for $(G,\beta)$. 
\end{lemma}

\begin{proof}
By construction $\beta_{\vec\C_1}=\beta$, so it only remains to prove that $\vec\C_1$ is an oriented chord diagram for $G$.
Consider a vertex $v\in V-c$. By definition $\delta_c(v,c)=1$ if and only if $vc\in E$. So the chord $v$ separates
$h$ from $t$ if and only if $vc\in E$, as required.
\end{proof}

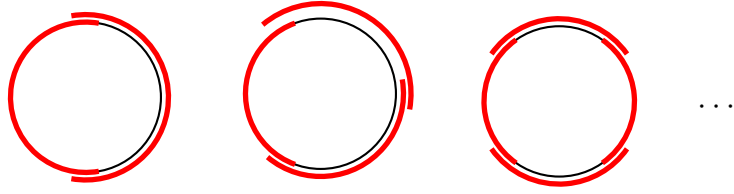
\begin{figure}
\begin{tikzpicture}
		\MakeCircle{1}{12}
		\draw ++(80:1) [color=red, line width = 2] arc (80:280:1);
		\draw ++(-100:1.1) [color=red, line width = 2] arc (-100:100:1.1);
\end{tikzpicture}
\hspace{5mm}
\begin{tikzpicture}
		\MakeCircle{1}{12}
		\draw ++(110:1) [color=red, line width = 2] arc (110:250:1);
		\draw ++(230:1.1) [color=red, line width = 2] arc (230:370:1.1);
		\draw ++(-10:1.2) [color=red, line width = 2] arc (-10:130:1.2);
\end{tikzpicture}
\hspace{5mm}
\begin{tikzpicture}
		\MakeCircle{1}{12}
		\draw ++(125:1) [color=red, line width = 2] arc (125:235:1);
		\draw ++(-55:1) [color=red, line width = 2] arc (-55:55:1);
		\draw ++(35:1.1) [color=red, line width = 2] arc (35:145:1.1);
		\draw ++(215:1.1) [color=red, line width = 2] arc (215:325:1.1);
\end{tikzpicture}
\hspace{5mm}
\raisebox{1cm}{$\cdots$}
\caption{Minimal covers}
\label{fig:covers}
\end{figure}

We are interested in the combinatorial properties of oriented chord diagrams 
as opposed to the specific topology; we can encode this combinatorial information
by traversing the perimeter of the circle in a clockwise direction and recording 
for each end of a chord that we encounter both the name of the chord and whether 
that end is a head or a tail. We consider two oriented chord diagrams to be {\em equivalent}
if they have the same such encodings.
\begin{lemma} \label{lemma:unique}
Let $\beta$ be a solution for the Naji system of a graph $G$. If $\beta$ is chordal and $G$ is connected,
then, up to equivalence, there is a unique oriented chord diagram for $(G,\beta)$.
\end{lemma}

\begin{proof}
We may  assume that $G$ has at least two vertices and, hence, that $G$ has a vertex $c$ such that 
$G-c$ is connected. Inductively we may assume that there is a unique oriented chord diagram $\vec\C$
for $(G-c,\beta[V(G-c)])$. Since $\beta$ is chordal, there is an oriented chord
diagram $\vec\C_1$ of $G$ respecting $\beta$; let $h$ and $t$ denote the head and tail of $c$ in $\vec\C_1$
respectively. By uniqueness, we may assume that $\vec\C_1$ contains $\vec\C$.  Hence, 
\[t\in T(G-c,\beta[V(G-c)]) \text{ and } h\in H(G-c,\beta[V(G-c)]).\] 

\begin{claim} 
$T(G-c,\beta[V(G-c)])$ and $H(G-c,\beta[V(G-c)])$ are both open arcs neither of which contains an end of any chord.
\end{claim}

\begin{proof}[Proof of claim.]
It suffices to prove the claim for $H(G-c,\beta[V(G-c)])$. 
Since $H(G-c,\beta[V(G-c)])$ is the intersection of finitely many open arcs it is an open set with finitely many components.
Since, for each $v\in V(G-c)$, the arc $H(G-c,\beta[V(G-c)],v)$ does not contain an end of $v$, the set
$H(G-c,\beta[V(G-c)])$ does not contain the end of a chord.
Let $A$ be a maximal closed arc disjoint from $H(G-c,\beta[V(G-c)])$ and let
$C$ denote the set of chords in $\vec C$ having an end in $A$.
Since the ends of $A$ are ends of chords, the set $C$ is non-empty. Moreover,
for each chord $v\in C$, since the arc  $H(G-c,\beta[V(G-c)],v)$ contains $H(G-c,\beta[V(G-c)])$, 
both ends of $v$ are contained in $A$. Then, since $G-c$ is connected, $H(G-c,\beta[V(G-c)])$ consists of a single arc.
\end{proof}

Now, to prove the uniqueness of $\vec\C_1$, it remains to prove that $T(G-c,\beta[V(G-c)]) \neq H(G-c,\beta[V(G-c)])$.
Since $G$ is connected, there is a neighbour $v$ of $c$. Note that $\delta_C(v,c)=1$, so
$T(G-c,\beta[V(G-c)],v)$ is disjoint from  $H(G-c,\beta[V(G-c)],v)$ and, hence,
\[T(G-c,\beta[V(G-c)]) \neq H(G-c,\beta[V(G-c)])\] as required.
\end{proof}

If we cannot extend $\vec\C$ to an oriented 
chord diagram of $(G,\beta)$ then, by Lemma~\ref{lemma:extend}, 
one of $T(\beta,\vec\C)$ and $H(\beta,\vec\C)$ is empty; by possibly reorienting $\beta$ at $c$
we may assume that $H(\beta,\vec\C)=\emptyset$. 
For a subset $X\subseteq V-c$ we let $H(\vC,\beta,X)$ denote the intersection of
the arcs $H(\vC,\beta,v)$ taken over all $v\in X$.
Consider a minimal subset $X$ of $V-c$ such that
$H(\beta,\vec\C,X)$ is empty.

Note that a collection of arcs has empty intersection if and only if the union 
of their complements covers the circle.
Figure~\ref{fig:covers} show some minimal configurations of arcs that cover the circle;
the following result shows that this list of examples is exhaustive.
In this result $\bZ_k$ denotes the set of integers modulo $k$.
\begin{lemma}\label{minimal:cover}
Let $\cA$ be a finite collection of closed arcs of a circle whose union covers the circle.
Then there is a sequence $(A_1,\ldots,A_k)$ of arcs in $\cA$ such that
$(A_1,\ldots,A_k)$ covers the circle and such that, for each $i\in \bZ_k$, the arc 
$A_i$ has a non-empty intersection with the arcs $A_{i-1}$ and $A_{i+1}$ but is disjoint from all other arcs.
\end{lemma}

\begin{proof} We may assume that $\cA$ is a minimal cover. Consider the set $P$ of points on the circle
that are in exactly one arc of $\cA$. Now $P$ partitions into a finite collection $(P_1,\ldots,P_k)$ of open arcs, which we number
according to their order around the circle. For each $i\in \{1,\ldots,k\}$, we let 
$A_i$ denote the arc containing $P_i$. By the minimality of $\cA$, the arcs $A_1,\ldots,A_k$ are distinct and
$\cA=\{A_1,\ldots,A_k\}$. By construction, for each $i\in \bZ_k$, the arc 
$A_i$ has a non-empty intersection with the arcs $A_{i-1}$ and $A_{i+1}$ but is disjoint from all other arcs.
\end{proof}

We are now ready to prove Lemma~\ref{obstructions}, which states that if $\beta$ is a solution is a non-chordal solution of the Naji system of a graph $G$, then there is an obstruction that is isomorphic to the Claw or $K_4$.

\begin{proof}[Proof of Lemma~\ref{obstructions}.]
We may assume that no proper induced subgraph of $G$ is both connected and an obstruction.
Let $c$ be a vertex of $G$ such that $G-c$ is connected, and let
$\vC$ be an oriented chord diagram for $(G-c,\beta[V(G-c)])$.

\medskip

\noindent
{\bf Case 1:}{ \it There are distinct vertices $u$ and $v$ of $G-c$ such that either
$H(\vC,\beta,\{u,v\})$ or $T(\vC,\beta,\{u,v\})$ is empty.}

\medskip

We may assume that $H(\vC,\beta,\{u,v\})$ is empty.
By possibly reorienting $\vC$ and $\beta$ on vertices in $V(G-c)$, we may assume that, for each vertex $w \in V(G-c)$,
the arc $H(\vC,\beta,w)$ is to the right of $w$ and hence that $\beta(v,c)= \beta(u, c) = 0$.
Let $P=(v_0,v_1,\ldots,v_k)$ be a shortest path from $c$ to $\{u,v\}$ in $G$; by symmetry, we may assume that $v_k=u$.
Note that $\beta(v,c)=0$ and, since the arcs to the right of $u$ and $v$ are disjoint, we have $\beta(v,u)=1$.
Then, by Lemma~\ref{strong}, the vertex $v$ must have a neighbour in $P$. 
Since $P$ is a shortest path from $c$ to $\{u,v\}$,  it must be the case that $v$ is adjacent to $v_{k-1}$.
Now we consider three cases $k=1$, $k>2$, and $k=2$.

First suppose that $k=1$. Summing  equations
$\NS_1(c,u)$, $\NS_1(c,v)$, and $\NS_3(c,u,v)$ gives
 $\beta(u,c)+\beta(v,c) + \beta(u,v) + \beta(v,u)=1$, which is a contradiction since
 $\beta(u,c)=\beta(v,c)=0$ and $\beta(u,v)=\beta(v,u)=1$.
 
 Now suppose that $k>2$. The head of $v_1$ is either to the left of $u$ or the left of $v$; by symmetry we may assume that
 it is to the left of $u$. Then $\beta(u,c)\neq\beta(u,v_1)$, contrary to $\NS_2(u,c,v_1)$.
 
 Finally suppose that $k=2$. Let $H=G[\{c,u,v,v_1\}]$; note that $H$ is isomorphic to the Claw.
 Moreover, since $H-c$ is connected and $H(\vC,\beta,V(H-c))=\emptyset$, it follows from
 Lemma~\ref{lemma:unique} that $H$ is an obstruction, as required.

\medskip

\noindent
{\bf Case 2:}{ \it For any two distinct vertices $u$ and $v$ of $G-c$, both
$H(\vC,\beta,\{u,v\})$ and $T(\vC,\beta,\{u,v\})$ are nonempty.}

\medskip

Since $\beta$ is not chordal, either $H(\vC,\beta))$ or $T(\vC,\beta)$ is empty.
By possibly reorienting $\beta$ at $c$, we may assume that $H(\beta,\vec\C)=\emptyset$.
By possibly reorienting $\vC$ and $\beta$ on vertices in $V(G-c)$, we may assume that, for each vertex $v\in V(G-c)$,
the arc $H(\vC,\beta,v)$ is to the right of $v$ and hence that $\beta(v,c)=0$.
Now, by Lemma~\ref{minimal:cover}, if $X$ is a minimal subset of $V(G-c)$ 
such that $H(\beta,\vec\C,X)=\emptyset$, then $G[X]$ is an induced cycle. 
Let $C$ be an induced cycle in $G-c$ with $H(\beta,\vec\C,V(C))=\emptyset$.
Let the vertices of $C$ be $(v_1,\ldots,v_k)$ in that cyclic order.
Among all such choices of $C$, if possible, we will take $C$ containing a neighbour of $c$.

\begin{newclaim} 
The vertex $c$ is either adjacent to every vertex in $C$ or adjacent to no vertex in $C$.
\end{newclaim}

\begin{proof}[Proof of claim.]
Suppose not, then, up to symmetry, we may assume that $c$ is adjacent to $v_1$ but not $v_2$.
Choose $j\in\{2,\ldots,k\}$ maximum such that $c$ is adjacent to none of $\{v_2,\ldots,v_j\}$.
Note that $v_1$ is adjacent to $v_2$ and $c$ but $cv_2\not\in E$. Moreover,
$\beta(v_1,v_2)=0$, $\beta(v_2,v_1)=1$, and $\beta(v_2,c)=0$. Therefore, by $\NS_3(v_1,v_2,c)$,
we have $\beta(c,v_2)=0$. Similarly we have
$\beta(c,v_{j-1})=1$. In particular $\beta(c,v_1)\neq\beta(c,v_{j-1})$, contradicting
Lemma~\ref{strong}.
\end{proof}

\begin{newclaim} 
If $c$ is adjacent to every vertex in $C$, then $G$ has a $K_4$-obstruction.
\end{newclaim}

\begin{proof}[Proof of claim.]
First consider the case that $C$ has three vertices. Then $G[V(C)\cup\{c\}]$ 
is isomorphic to $K_4$. Moreover, it follows from Lemma~\ref{lemma:unique} that $G[V(C)\cup\{c\}]$
is an obstruction, as required. Hence we may assume that $C$ has at least four vertices.
Let $\beta'=\beta+\delta_c$. Note that $\beta'(v_i,c)=1$ for each vertex $v_i$ of $C$.
In particular, this implies that $T(\C,\beta,v_1)$ and $T(\C,\beta,v_3)$ are disjoint, 
contrary to the hypotheses of this case.
\end{proof}

Henceforth we may assume that $c$ has no neighbours in $C$. Note that in this case 
$\delta_c(v_i,c)=0$ for each $i\in \bZ_k$ and hence $T(\vC,\beta)=\emptyset$. Therefore we are free
to reorient $\beta$ at $c$, however, when we reorient $\beta$ at $c$ we should also
reorient the chords adjacent to $c$ so that we keep the property that, for each $v\in V-c$,
the arc $H(\vC,\beta,v)$ lies to the right of $v$.
 
Let $v$ be a neighbour of $c$.  Since $\beta[V(G - c)]$ is chordal, by $\NS_1$ $v$ is also a neighbour of $C$.
For each $i\in \bZ_k$ we let $A_i$ denote the arc
$H(\vC,\beta,V(C)-\{v_i\})$. These arcs are disjoint and $k\ge 3$,
so one of these arcs, say $A_i$, lies either entirely to the right of $v$ or entirely to the left of $v$.
By possibly reorienting $\beta$ at $c$ and at each of its neighbours, we may assume that
$A_i$ lies to the left of $v$.
Thus  $H(\vC,\beta,(V(C)-\{v_i\})\cup\{v\})$  is empty.
Then, by Lemma~\ref{lemma:unique}, there is an induced cycle $C'$ in $G[(V(C)-\{v_i\})\cup\{v\}]$
such that $H(\vC,\beta,V(C'))$ is empty.
Note that $v\in V(C')$ and that this contradicts our initial choice of $C$.
\end{proof}

\section{Inducing a split}\label{section:splits}

In this section we complete the proof of Naji's Theorem by proving Lemma~\ref{induced:split}, showing that
any split in either a $K_4$-obstruction or in a  Claw obstruction will extend to a split in the original graph.

Let $(X_0,Y_0)$ be a split in an induced subgraph $H$ of $G$.
We say that $(X_0,Y_0)$ {\em induces a split} in $G$ if there 
is a split $(X,Y)$ in $G$ with $X_0\subseteq X$ and $Y_0\subseteq Y$.

We will start with Claw-obstructions, which are a little easier to deal with;  this proof is similar to that of Lemma~4.1 in~[\ref{GEL}].
\begin{lemma}\label{part1}
Let $\beta$ be a solution to the Naji system for a graph $G$ and let $H$ be a Claw-obstruction in 
$G$. Then each split in $H$ induces a split in $G$.
\end{lemma}

\begin{proof}
Consider an induced claw $G[\{x,a,b,c\}]$ in $G$ where $x$ is the vertex of degree $3$.
Summing $\NS_3(x,a,b)$, $\NS_3(x,b,c)$, and $\NS_3(x,c,a)$ gives
$$ (\beta(a,b)+\beta(a,c)) + (\beta(b,a)+\beta(b,c)) +(\beta(c,a)+\beta(c,b)) =1.$$
Therefore either one or three of $\beta(a,b)+\beta(a,c)$, $\beta(b,a)+\beta(b,c)$, and
$\beta(c,a)+\beta(c,b)$ is equal to $1$. Given three pairwise non-intersecting  chords $a'$, $b'$, and $c'$ in an oriented chord diagram $\vC$, we have
$\beta_{\vC}(a',b')+\beta_{\vC}(a',c')=1$ if an only if $a'$ separates $b'$ from $c'$. However, if $a'$ separates $b'$ from $c'$ then neither
$b'$ nor $c'$ separate the other two chords. Therefore, if
$\beta(a,b)+\beta(a,c)=1$, $\beta(b,a)+\beta(b,c)=1$, and
$\beta(c,a)+\beta(c,b)=1$, then $G[\{x,a,b,c\}]$ is an obstruction.
It is left to the reader to verify that $\beta[\{x,a,b,c\}]$ is chordal when exactly one of
$\beta(a,b)+\beta(a,c)$, $\beta(b,a)+\beta(b,c)$, and
$\beta(c,a)+\beta(c,b)$ is equal to $1$.

We denote the set of  neighbours of a vertex $v$ by $N(v)$.
We first prove the following two claims, analogous to Claims 4.1.1 and 4.1.2 in~[\ref{GEL}].
\begin{claim} \label{claim:claw1}
Let $G[\{x,a,b,c\}]$ be a Claw-obstruction where $x$ is the vertex of degree three and let
$X=N(a)\cap N(b)\cap N(c)$.
Then $a$, $b$ and, $c$ are in distinct components of $G - X$.
\end{claim}
\begin{proof}
Suppose otherwise and let $P$ be a shortest path in $G-X$ connecting two of $a$, $b$, and $c$.   
By symmetry we may assume that $P$ contains $a$ and $b$.
Since $G[\{x,a,b,c\}]$ is a Claw-obstruction, we have $\beta(c, a) \neq \beta(c,b)$.
Then, by Lemma~\ref{strong}, the vertex $c$ must have a neighbour, say $z$, in $P$.
However, by the minimality of $P$, both $za$ and $zb$ must be edges in $P$.
But then $z\in X$, which is not possible for vertices of $P$.
\end{proof}

Suppose that $V(H)=\{x,a,b,c\}$ where $x$ is the vertex of degree three in $H$ and let $X=N(a)\cap N(b)\cap N(c)$.
Let $X_a$ (respectively $X_b$ and $X_c$) denote the set of vertices that are in the same component of $G -  X$ as $a$ (respectively $b$ and $c$).
\begin{claim} \label{claim:claw2}
If $d\in X_a\cup X_b\cup X_c$ is a vertex with a neighbour in $X$, then $X$ is contained in $N(d)$. 
\end{claim}
\begin{proof}
Up to symmetry we may assume that $d\in X_c$.
Let $x'\in X$ be a neighbour of $d$. Note that $G[\{x',a,b,d\}]$ is a Claw.
By Claim \ref{claim:claw1}, the vertex $a$ is not in the same component of $G- X$ as $c$.  Hence, by Lemma \ref{strong}, we have  $\beta(a, c) = \beta(a, d)$.  By a 
symmetric argument $\beta(b, c) = \beta(b, d)$.   Since $H$ is a Claw-obstruction, $\beta(a,b) + \beta(a, c) = 1$ and $\beta(b, a) + \beta(b, c) = 1$. So
$\beta(a,b) + \beta(a, d) = 1$ and $\beta(b, a) + \beta(b, d) = 1$, and, hence, $G[\{x',a,b,d\}]$ is a Claw-obstruction.
Then, by Claim~\ref{claim:claw1}, it must be the case that $X\subseteq N(d)$, as required.
\end{proof}

Now consider a split $(A,B)$ in $H$; up to symmetry we may assume that 
 $A=\{a,b\}$ and $B=\{x,c\}$. Let $A'=X_a\cup X_b$ and $B'=V(G)-A'$. Note that $A\subseteq A'$, $B\subseteq B'$, and,
 by Claim~\ref{claim:claw2}, $(A',B')$ is a split in $G$.
\end{proof}

We now complete the proof of Lemma~\ref{induced:split} by showing that splits
in $K_4$-obstructions induce splits in the full graph.
\begin{lemma}\label{part2}
Let $\beta$ be a solution to the Naji system for a graph $G$ and let $H$ be a $K_4$-obstruction in 
$G$. Then each split in $H$ induces a split in $G$.
\end{lemma}

\begin{proof}
For an edge $e=uv$ of $G$ the equation $\NS_1(u,v)$ implies that
exactly one of $\beta(u,v)$ and $\beta(v,u)$ is one.
So we can construct an orientation $\vG$ of $G$ such that 
$v$ is the head of $e$ if and only if $\beta(u,v)=1$.
Reorienting $\beta$ at a vertex $x$ has the effect of 
changing the orientations on all edges incident with $x$ 
and leaving the other edge orientations as they were.  

Consider a subgraph $H_0$ of $G$ that is isomorphic to $K_4$
and let $x\in V(H_0)$. We can reorient $\beta$ so that $H_0-x$
is a directed cycle in $\vG$ and so that at least two of the three edges of
$H_0$ incident with $x$ have $x$ as their tail. It is easy to verify that,
if the third edge has $x$ as head, then $\beta[V(H_0)]$ is chordal, while,
if that edge has $x$ as its tail, $H_0$ is an obstruction. 

Consider a $4$-cycle $C$ in $G$. We refer to $C$ as {\em odd} (respectively {\em even})
if we encounter an {\em odd} (respectively {\em even}) number of forward arcs
when we traverse $C$ in $\vG$; since $C$ has an even number of edges it does not 
matter which direction we traverse $C$. It is now easy to verify that $H_0$ is an obstruction
if and only if every $4$-cycle in $H_0$ is odd.   

\begin{claim}\label{key:claim}
Let $H_0=G[\{a,b,c,d\}]$ be a subgraph of $G$ isomorphic to $K_4$ and let $P$ be a path with 
distinct ends $a$ and $b$ in $H_0$ such that $V(P)\cap V(H_0) =\{a,b\}$ and $E(P)\cap E(H_0) = \emptyset$.
If $P\cup H_0$ is an induced subgraph of $G$, then the $4$-cycle $(a,c,b,d,a)$ of $G$ is even.
\end{claim}

\begin{proof}[Proof of claim.]
Suppose that the vertices of $P$ are $(v_0,v_1,\ldots,v_k)$, in that order, from $a$ to $b$. 
First consider the case that $k=2$. Adding the equations
$\NS_3(a,c,v_1)$, $\NS_3(b,c,v_{1})$, $\NS_3(a,d,v_1)$, $\NS_3(b,d,v_{1})$,
$\NS_1(a,c)$, and $\NS_1(b,d)$ gives
$ \beta(c,a)+\beta(a,d)+\beta(d,b) + \beta(b,c) =0,$
and hence the $4$-cycle $(a,c,b,d,a)$ of $G$ is even.
So we may assume that $k>2$.
By Lemma~\ref{strong}, we have $\beta(c,v_1)+\beta(c,v_{k-1})=0$ and $\beta(d,v_1)+\beta(d,v_{k-1})=0$.
Now add these two equations together with the equations $\NS_3(a,c,v_1)$, $\NS_3(b,c,v_{k-1})$, $\NS_3(a,d,v_1)$, $\NS_3(b,d,v_{k-1})$, \\
$\NS_2(v_1,c,d)$, $\NS_2(v_{k-1})$, $\NS_1(a,c)$, and $\NS_1(b,d)$ to obtain
$ \beta(c,a)+\beta(a,d)+\beta(d,b) + \beta(b,c) =0,$
and hence the $4$-cycle $(a,c,b,d,a)$ of $G$ is even.
\end{proof}

\begin{claim}\label{transitive}
Let $H_0=G[\{a,b,c,d\}]$ be a $K_4$-obstruction.
If $a'\in V(G)-\{a,b,c,d\}$ is a vertex that is adjacent to $b$, $c$, and $d$ but not $a$, then
$G[\{a',b,c,d\}]$ is a $K_4$-obstruction.
\end{claim}

\begin{proof}[Proof of claim.] Consider an arbitrary $4$-cycle $C$ of $H_0$.
Up to symmetry we may assume that $C$ is $(a,b,c,d,a)$. Since $H_0$ is a
$K_4$-obstruction,
$$ \beta(a,b)+\beta(b,c)+\beta(c,d)+\beta(d,a) = 1.$$
Adding the equations $\NS_3(d,a,a')$, $\NS_3(b,a,a')$, $\NS_1(a,b)$, and $\NS_1(a',b)$
to this equation gives
 $$ \beta(a',b)+\beta(b,c)+\beta(c,d)+\beta(d,a') = 1.$$
So each $4$-cycle of $G[\{a',b,c,d\}]$ is odd, as required.
\end{proof}

Choose maximal disjoint vertex-sets $(X_a,X_b,X_c,X_d)$ such that
\begin{itemize}
\item[(i)] each set $(X_a,X_b, X_c,  X_d)$ contains a vertex of $H$, and
\item[(ii)] for each $a\in X_a$, $b\in X_b$, $c\in X_c$, and $d\in X_d$
the subgraph $G[\{a,b,c,d\}]$ is a $K_4$-obstruction.
\end{itemize}
Since $H$ is a $K_4$-obstruction, every 4-cycle of $H$ is odd, so such sets exist. Let $X=X_a\cup X_b\cup X_c\cup X_d$.

\begin{claim}\label{neighbours}
For each $v\in V(G)-X$, either
\begin{itemize}
\item $v$ is adjacent to vertices in at most one of the sets
$(X_a,X_b,X_c,X_d)$, or
\item $v$ is adjacent to every vertex in $X$.
\end{itemize}
\end{claim}

\begin{proof}[Proof of Claim.]
Suppose otherwise that $v$ has neighbours in at least two of the sets
$(X_a,X_b,X_c,X_d)$, but that $v$ is not adjacent to every vertex in $X$.
By Claim~\ref{key:claim}, $v$ cannot have exactly two neighbours in any
$K_4$-obstruction. It follows that $v$ has neighbours in at least three of
$(X_a,X_b,X_c,X_d)$. Now, up to symmetry we can choose elements
$a\in X_a$, $b\in X_b$, $c\in X_c$, and $d\in X_d$ such that
$v$ is adjacent to $b$ and $c$ but not $a$. By Claim~\ref{key:claim},
$v$ is also adjacent to $d$. By changing our choice of $d\in X_d$ (respective
$b\in X_b$ and $c\in X_c$) and applying Claim~\ref{key:claim},
we have that $v$ is adjacent to each vertex in $X_d$ (respectively $X_b$ and $X_c$).
Now, for any $b'\in X_b$, $c'\in X_c$ and $d'\in X_d$, by Claim~\ref{transitive}, 
we have that $G[\{a,b',c',d'\}]$ is a $K_4$-obstruction.
Thus $(X_a\cup\{a\},X_b,X_c,X_d)$ satisfies $(ii)$, but this contradicts the 
maximality of our collection $(X_a,X_b,X_c,X_d)$.
\end{proof}

Let $Y$ denote the set of all vertices in $V(G)-X$ that are adjacent to every vertex in $X$.

\begin{claim}\label{claim:split}
Each component in $G-(X\cup Y)$ has neighbours in at most one of the sets
$(Y,X_a,X_b,X_c,X_d)$.
\end{claim}

 \begin{proof}[Proof of Claim.]
 By Claims~\ref{key:claim} and~\ref{neighbours}, no component of $G-(X\cup Y)$ has neighbours in two of the sets
 $(X_a,X_b,X_c,X_d)$. Suppose that there is a component of $G-(X\cup Y)$ with neighbours in both $X$ and $Y$.
 Consider a shortest path $P=(v_0,v_1,\ldots,v_k)$ such that $v_0$ has a neighbour in $X$ and $v_k$ has a neighbour in $Y$.
 Suppose that $a\in X$ is a neighbour of $v_0$ and $a'\in Y$ is a neighbour of $v_k$. By symmetry we may assume that 
 $a\in X_a$. By the maximality of $(X_a,X_b,X_c,X_d)$, there exist $b\in X_b$, $c\in X_c$, and $d\in X_d$ such that
 $G[\{a',b,c,d\}]$ is not a $K_4$-obstruction. By possibly reorienting $\beta$ at $b$, $c$, and $d$ we may assume that 
 the edges $ab$, $ac$, and $ad$ each have $a$ as their head and by possibly reorienting $\beta$ at $a'$ we may assume that
 at least two of the edges $a'b$, $a'c$, and $a'd$ have $a'$ as their head. Up to symmetry we may assume that $a'$ is the head
 of both $a'b$ and $a'c$. Since $G[\{a,b,c,d\}]$ is a $K_4$-obstruction but $G[\{a',b,c,d\}]$ is not, $a'$ must be the tail of $a'd$.
 However, then the $4$-cycle $(a,b,a',d,a)$ is odd, contrary to Claim~\ref{key:claim}. 
 \end{proof}
 
 Now consider a split $(A,B)$ in $H$; up to symmetry we may assume that 
 $A\subseteq X_a\cup X_b$ and $B\subseteq X_c\cup X_d$. Let $A'$ denote
 the union of $X_a$, $X_b$, together with the set of all
 vertices in components of $G-(X\cup Y)$ that have a neighbour in $X_a\cup X_b$.
 Let $B'=V(G)-A'$. Note that $A\subseteq A'$, $B\subseteq B'$, and,
 by Claim~\ref{claim:split}, $(A',B')$ is a split in $G$.
 \end{proof}

\section{Acknowledgements}
We thank the referees for their helpful comments.

\section*{References}
        
\newcounter{refs}
       
\begin{list}{[\arabic{refs}]}%
{\usecounter{refs}\setlength{\leftmargin}{10mm}\setlength{\itemsep}{0mm}}

\item \label{BO}
A. Bouchet,
Circle graph obstructions,
J. Combin. Theory, Ser. B, { 60} (1994), pp. 107-144.

\item \label{Bouchet1987}
A. Bouchet,
Reducing prime graphs and recognizing circle graphs,
Combinatorica { 7} (1987), pp. 243-254.

\item\label{FRAY}
H. de Fraysseix,
Local complementation and interlacement graphs,
Discrete Math. 33 (1981), pp. 29-35.

\item\label{GASSE}
E. Gasse,
A proof of a circle graph characterization,
Discrete Math. 173 (1997), pp. 273-283.

\item\label{GEL}
J. Geelen, B. Gerards,
Characterizing graphic matroids by a system of linear equations,
J. Comb. Theory, Ser. {B}, 103 (2013), pp. 642-646.

\item\label{NG-ORIG}
W. Naji,
{\em Graphes des cordes: une caract\'{e}risation et ses applications} (Th\' ese), Grenoble, 1985.

\item\label{TLN}
L. Traldi,
Notes on a theorem of Naji,
Discrete Math. 340 (2017), pp. 3217-3234.

\end{list}

\end{document}